\documentclass[11pt,a4paper]{article}

\usepackage{amsmath}
\usepackage{fullpage}
\usepackage{graphicx}
\usepackage{amssymb}
\usepackage{stmaryrd}

\title{A note on a Cayley graph of $\mathfrak{S}_n$.}

\author{
   Guillaume Chapuy, 
   \\ \small\it CNRS - LIAFA, Universit\'e Paris 7, 
   \\ \small\it Paris, France.
   \and
   Valentin F\'eray, \\ \small\it CNRS - LaBRI, Universit\'e Bordeaux 1,
   \\ \small\it Talence, France.
   }

\newtheorem{theorem}{\bf Theorem}[section]
\newtheorem{corollary}[theorem]{\bf Corollary}

\newtheorem{proposition}[theorem]{\bf Proposition}

\newenvironment{proof}%
{\noindent{\bf Proof.}\ }%
{\hfill$\Box$\par\bigskip}%

\begin{document}

\date{February 14, 2012}
\maketitle

\begin{abstract}
  Recently in graph theory several authors have studied the spectrum of the
  Cayley graph of the symmetric group $\mathfrak{S}_n$ generated by the
  transpositions $(1,i)$ for $2\leq i\leq n$. Several conjectures were made and
  partial results were obtained.

  The purpose of this note is to point out that,
  as mentioned also by P. Renteln,
  this problem is actually
  already solved in another context.
  Indeed it is
  equivalent to studying the spectrum of so-called Jucys-Murphy elements in the
  algebra of the symmetric group,
  which is well understood.
  The aforementioned conjectures are 
  direct consequences of the existing theory.
  We also present a related result from P. Biane, giving an asymptotic description
  of this spectrum.
  
  We insist on the fact that this note does not contain any new results,
  but has only been written to convey the information from the algebraic combinatorics
  community to graph theorists.
\end{abstract}

\maketitle

\section{Jucys-Murphy elements.}

We let $n\geq 1$ be an integer and $\mathfrak{S}_n$ be the symmetric group on
$\{1,2,\dots,n\}$. We let $\mathcal{P}(n)$ be the set of partitions of $n$ and
we note $\lambda \vdash n$ if $\lambda \in \mathcal{P}(n)$.
The irreducible representations of $\mathfrak{S}_n$ are canonically indexed by elements of
$\mathcal{P}(n)$, and we denote by $V_\lambda$ the irreducible module
associated with the partition $\lambda$.
The regular representation of $\mathfrak{S}_n$ is decomposed into irreducible
submodules as follows \cite[Proposition 1.10.1]{Sagan}:
\begin{eqnarray}\label{eq:repreg}
\mathbb{C}[\mathfrak{S}_n] = \bigoplus_{\lambda \in \mathcal{P}(n)}
V_{\lambda}^{f_\lambda},
\end{eqnarray}
where $f_\lambda := \mathrm{dim} V_\lambda$.\\

A partition $\lambda \vdash n$ is represented by its Ferrers diagram as on
Figure~\ref{fig:ferrers}(a).
A \emph{standard Young tableau (SYT)} of shape $\lambda$ is a filling of the
Ferrers diagram of
$\lambda$ with the elements $\{1,2,\dots,n\}$ in such a way that elements
increase along rows and columns (in particular all elements are distinct, and
each element appears exactly once). See Figure~\ref{fig:ferrers}(b).
We let $\mathcal{T}(\lambda)$ we the set of SYT
of shape $\lambda$.
A well known result asserts~\cite[Theorem 2.6.5]{Sagan} that
$f_\lambda=\mathrm{card}\ \mathcal{T}(\lambda)$.

\begin{figure}\label{fig:ferrers}
  \begin{center}
\includegraphics[scale=1]{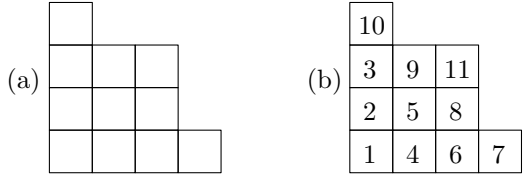}
  \caption{(a) The Ferrers diagram of the partition $\lambda=(4,3,3,1)$ (we use
  the French notation); (b) A standard Young Tableau.}
\end{center}
\end{figure}

Let $\lambda$ be a partition, and let $\square$ be a box of the Ferrers diagram of
$T$. Let $x(\square)$ and $y(\square)$ be its abscissa and ordinate, respectively,
and define the \emph{content} of the box as $c(\square)=y(\square)-x(\square)$.
If $T\in\mathcal{T}(\lambda)$ is a SYT, and if $i$ is an element of $\{1,2,\dots,
n\}$, we define $c_T(i)$ as the content of the box in which the label $i$
appears in the tableau $T$. For example on
Figure~\ref{fig:ferrers}(b) one has $c_T(2)=1$, $c_T(5)=0$ and $c_T(6)=-2$.

The \emph{Jucys-Murphy elements} $J_2,\dots,J_n$ are elements of the group
algebra $\mathbb{C}[\mathfrak{S}_n]$ introduced separately by A. Jucys \cite{Jucys1966}
and G. Murphy \cite{Murphy1981}. In recent years they have proved to be very
efficient tools in the study of the representations of the symmetric
group~\cite{OkVe1996} (see also \cite{Rep}).
They are defined by:
$$J_i = (1, i) + (2,i) + \dots + (i-1,i).$$
As shown by the following theorem~\cite[equation (12)]{Jucys1974}, the action of the $J_i's$
on the irreducible modules $V_\lambda$ is diagonal, and the eigenvalues
have a combinatorial description in terms of contents
(see also \cite[Theorem 3.7]{RandomShufflesRepresentations},
where the Cayley graph of $S_n$ is explicitly mentioned):
\begin{theorem}[\cite{Jucys1974}]\label{thm:JM}
  Let $\lambda \in \mathcal{P}(n)$.
  Then there exists a basis $(v_T)_{T\in\mathcal{T}(\lambda)}$ of the irreducible 
  module $V_\lambda$, indexed by SYT's of shape $\lambda$, such that for all 
  $i \in \{2,\dots,n\}$, one has:
  $$ J_i v_T = c_T(i) v_T .$$
\end{theorem}

\section{A Cayley Graph}

As in \cite{FriedmanCayleyGraphs,CayleyJMConjectures,MoharSpectrum},
we consider the Cayley graph $G_n$ on $\mathfrak{S}_n$ generated by
the transpositions $(1,i)$ for $2\leq i\leq n$. The \emph{spectrum} of this
graph is defined as the spectrum of its adjacency matrix. It (evidently) coincides with the
spectrum of the left multiplication by the Jucys-Murphy element $J_n$ on the group
algebra $\mathbb{C}[\mathfrak{S}_n]$
(this link and the fact that it implies that the aforementioned spectrum
is integral
are written explicitly in the introduction of paper
\cite{renteln2011distance}).
From \eqref{eq:repreg} and Theorem~\ref{thm:JM} we immediately obtain:
\begin{corollary}\label{cor}
  The spectrum of $G_n$ contains only integers. The
  multiplicity $\mathrm{mul}(k)$ of an integer $k\in\mathbb{Z}$ is given by:
  $$
  \mathrm{mul(k)} = \sum_{\lambda\in \mathcal{P}(n)}
		     f_\lambda I_\lambda(k)
  $$
  where $I_\lambda(k) = \mathrm{card} \{T\in\mathcal{T}(\lambda), c_T(n)=k\}$
  is the number of standard Young tableaux of shape $\lambda$ in which the
  integer $n$ appears in a box of content $k$, and $f_\lambda$ is the number of
  SYT of shape $\lambda$.
\end{corollary}

\noindent We list some direct remarks and consequences: \vspace{-3mm}
 \begin{enumerate}\setlength{\itemsep}{0pt}\setlength{\parskip}{0pt}
   \item All eigenvalues are integers, as conjectured 
       in~\cite[Conjecture 2.14]{CayleyJMConjectures}.
   \item The spectrum of $G_n$ is contained in
     $\{-(n-1),\dots,n-1\}.$ Moreover one has
     $\mathrm{mul}(-n+1)=\mathrm{mul}(n-1)=1$.
   \item Unless $n= 2$ or $n=3$, one has $\mathrm{mul}(0)\neq 0$ (as proved in \cite{MoharSpectrum}).
       Indeed unless $n=2$ or $n=3$, there
     exists a SYT of size $n$ in which $n$ appears on the main diagonal.
     More generally, the spectrum is given by:
     \begin{eqnarray*}
     \{k,\mathrm{mul}(k)\neq 0\}&=& \{-(n-1),\dots,n-1\}\setminus\{0\} \mbox{ if }
     n\in\{2,3\},\\
     &&\{-(n-1),\dots,n-1\} \mbox{ if }   n>3.
   \end{eqnarray*}
   \item Let $1\leq l \leq n$, and consider the ``hook-shaped'' partition
     $\lambda^{(l)}=(n-l,1,1,\dots,1)$ of $n$. The dimension of this partition
     is $f_{\lambda^{(l)}}={n-1 \choose l}$, as can be seen directly or via the
     hook-length formula (\cite{HookLengthFormula1954} or~\cite[Theorem
     3.10.2]{Sagan}).
     Moreover, the number of standard Young tableaux $T$ of shape
     $\lambda^{(l)}$ such that
     $c_T(n)=l$ (i.e., such that $n$ appears in the topmost box) equals
     ${n-2 \choose l-1}$, since such a tableau is determined by the choice of
     the $l-1$ elements appearing between $1$ and $n$ on the left row of the
     tableau. Hence by Corollary~\ref{cor} one has:
     $$
     \mathrm{mul}(l) \geq {n-2 \choose l-1} {n-1 \choose l}.
     $$
     This improves the bound $\mathrm{mul}(l) \geq {n-2 \choose l-1}$ proved
     in~\cite{MoharSpectrum}. Similar arguments (transpose tableaux) show that the same
     bound holds for $\mathrm{mul}(-l)$.
    \item The bound given in point 4 could possibly be refined by taking more complicated
      tableaux than hook-shapes into account.
      Rather than pursuing in this direction,  we will show in the next section
      that almost all the eigenvalues are of order $O(\sqrt{n})$, and we will
      give a very precise description of the spectrum of $G_n$ in this range of
      values.
 \end{enumerate}

 \section{Semi-circle law}

 To be comprehensive on what is known about the spectrum of $J_n$,
 we present here an asymptotic result of P. Biane
 \cite[Theorem 1]{BianeJMSemiCircleLaw}:
 the spectral measure converges in
 distribution to the semicircle law.
 This provides a good
 description of the spectrum in the range $k=\Theta(\sqrt{n})$.

 The proof technique
 is standard and goes back to the beginning of random matrix theory.
 As it is quite short and elegant, we repeat it here.

 For
 $x\in\mathbb{R}$ we denote
 by $\delta_x$ the Dirac measure at $x$.
 \begin{proposition}
   The measure 
   $$\mathrm{sp}_n:=\frac{1}{n!} \sum_{k\in\mathbb{Z}} \mathrm{mul}(k)
   \delta_{\frac{k}{2\sqrt{n}}}$$
   converges in distribution to the semi-circle law, \emph{i.e.} for
   all $a<b$ we 
   have the convergence:
   $$
   \frac{1}{n!}\sum_{k\in[2a\sqrt{n},2b\sqrt{n}]\cap \mathbb{Z}} \mathrm{mul}(k) 
   \longrightarrow \frac{2}{\pi}\int_a^b \sqrt{1-\alpha^2} d\alpha
   $$
   when $n$ tends to infinity.
 \end{proposition}
 \begin{proof}[sketch]
     It is enough to prove (see, {\em e.g.}, \cite[Theorem 30.2]{BillingsleyProbMeasure})
     the convergence of moments, i.e. to
   prove that for any fixed $k\in\mathbb{Z}$ one has when $n$ tends to infinity:
   \begin{eqnarray}\label{moments}
     \frac{1}{n!}	\sum_{l\in\mathbb{Z}} \mathrm{mul}(l)
     \left(\frac{l}{2\sqrt{n}}\right)^k =   \frac{n^{-k/2}}{2^k n!} 
   \mathrm{\bf Tr } J_n^k \longrightarrow
   \frac{2}{\pi}\int_\mathbb{R} \alpha^k\sqrt{1-\alpha^2} d\alpha
   =\left\{\begin{array}{cc}0, & k \mbox{ odd} \\ \frac{1}{2^{k}}\mathrm{Cat}(p) & 
      k=2p.\end{array}\right.
    \end{eqnarray}
   where $\mathrm{Cat}(p):=\frac{(2p)!}{(p+1)!p!}$ is the $p$-th Catalan number.
   Let $\sigma\in\mathbb{S}_n$. The multiplication by $\sigma$ acts by
   permutation on the canonical basis of the group algebra
   $\mathbb{C}[\mathfrak{S}_n]$. Therefore $\mathrm{\bf Tr} \sigma$ equals the
   number of fixed points under this action, i.e.:
   $$
   \mathrm{\bf Tr } \sigma = \mathrm{card }\{\mu\in\mathfrak{S}_n, \sigma \mu = \mu\}
   =\left\{\begin{array}{cc}n!, & \sigma=id \\ 0 & \sigma\neq
     id.\end{array}\right.
   $$
   By developing the product $J_n^k= ((1\ n) + \dots + (n-1\ n) )^k$, this
   implies that $\frac{1}{n!}\mathrm{\bf Tr}J_n^k$ equals the number of
   $k$-uples $(i_1,\dots,i_k)\in \llbracket 1, n-1\rrbracket^k$ such that 
   $(i_1\ n) (i_2\ n)\ \dots\ (i_k \ n) =  id$.
   This number is $0$ if $k$ is odd (signature), so we assume that $k=2p$.
   Since $k$ is fixed and we are interested in large $n$ asymptotics, we only
   need to consider the cases where the set of values $I:=\{i_l, l=1,2,\dots,k\}$ has the
   largest cardinality. It is easily seen that each transposition $(i\ n)$ must appear an even
   number of times in the product, so the maximal case is $\mathrm{card}\ I = p$.
   In this case, we consider the pairing of elements of
   $\{1,2,\dots,k\}$ defined by pairing $l$ and $m$ if and only if
   $i_l=i_m$. It is convenient to represent this pairing by a diagram. For
   example, here is a drawing of this diagram in the case when
   $k=8$ and the product has the form $(i_1\ n) (i_2\ n) (i_2\ n) (i_4\ n)
   (i_4\ n) (i_1\ n) (i_7\ n) (i_7\ n) = id$:

  \begin{center}
\includegraphics[scale=1]{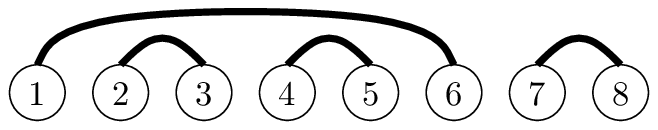}
\end{center}

   In general it is easily seen that this pairing is necessarily non-crossing. Conversely, there are
   exactly $(n-1)(n-2)\dots (n-p)$ ways to
   reconstruct a factorisation $(i_1\ n) (i_2\ n)\
   \dots\ (i_k \ n) =  id$  
   from one of the $Cat(p)$ non-crossing pairings of the set
   $\{1,2,\dots,2p\}$. 
   Therefore one has:
   \begin{eqnarray*}
     \frac{1}{n!}\mathrm{\bf Tr}\ J_n^k &=&
     (n-1)(n-2)\dots (n-p) Cat(p)  + O(n^{p-1})\\
	 &\sim& n^p Cat(p),
   \end{eqnarray*}
 and the convergence of moments \eqref{moments} is proved.
 \end{proof}

 \section*{Acknowledgements}
 We thank P. Renteln for pointing out references \cite{RandomShufflesRepresentations} and \cite{renteln2011distance}.

\small
\bibliographystyle{abbrv}
\bibliography{biblio1202.bib}

\end{document}